\documentclass[a4paper]{amsart}
\usepackage{amsmath,amsthm,amssymb,hyperref,mathrsfs}
\usepackage{aliascnt}
\usepackage{lmodern}
\usepackage[T1]{fontenc}

\usepackage[textsize=footnotesize]{todonotes}

\usepackage{xcolor}
\definecolor{dblue}{rgb}{0,0,0.70}
\hypersetup{
	unicode=true,
	colorlinks=true,
	citecolor=dblue,
	linkcolor=dblue,
	anchorcolor=dblue
}

\makeatletter
\expandafter\g@addto@macro\csname th@plain\endcsname{%
		\thm@notefont{\bfseries}
	}%
\expandafter\g@addto@macro\csname th@remark\endcsname{%
		\thm@headfont{\bfseries}
	}%
\makeatother

\newtheorem{theorem}
{Theorem}[section]	
\newtheorem*{theorem*}{Theorem}

\newaliascnt{lemma}{theorem}
\newtheorem{lemma}[lemma]{Lemma}

\aliascntresetthe{lemma}
\newtheorem*{lemma*}{Lemma}

\newaliascnt{proposition}{theorem}
\newtheorem{proposition}[proposition]{Proposition}
\aliascntresetthe{proposition}

\newaliascnt{corollary}{theorem}
\newtheorem{corollary}[corollary]{Corollary}
\aliascntresetthe{corollary}

\theoremstyle{remark}

\newaliascnt{remark}{theorem}
\newtheorem{remark}[remark]{Remark}
\aliascntresetthe{remark}
\newaliascnt{question}{theorem}
\newtheorem{question}[question]{Question}
\aliascntresetthe{question}
\newaliascnt{conjecture}{theorem}

\aliascntresetthe{conjecture}

\newtheorem*{question*}{Question}

\newaliascnt{definition}{theorem}
\newtheorem{definition}[definition]{Definition}
\aliascntresetthe{definition}

\newaliascnt{example}{theorem}

\aliascntresetthe{example}

\renewcommand{\restriction}{\mathbin\upharpoonright}

\newcommand{\axiom}[1]{\mathsf{#1}} 
\newcommand{\ZFC}{\axiom{ZFC}}

\newcommand{\DC}{\axiom{DC}}
\newcommand{\ZF}{\axiom{ZF}}

\newcommand{\Ord}{\mathrm{Ord}}

\newcommand{\GCH}{\axiom{GCH}}

\newcommand{\IS}{\axiom{IS}}
\newcommand{\HS}{\axiom{HS}}

\DeclareMathOperator{\cf}{cf}
\DeclareMathOperator{\dom}{dom}
\DeclareMathOperator{\rng}{rng}
\DeclareMathOperator{\supp}{supp}

\DeclareMathOperator{\sym}{sym}

\DeclareMathOperator{\fix}{fix}

\DeclareMathOperator{\id}{id}
\DeclareMathOperator{\aut}{Aut}

\DeclareMathOperator{\Add}{Add}

\DeclareMathOperator{\otp}{otp}

\newcommand{\forces}{\mathrel{\Vdash}}
\newcommand{\nforces}{\mathrel{\not{\Vdash}}}

\newcommand{\incompatible}{\mathrel{\bot}}
\newcommand{\compatible}{\mathrel{\|}}

\newcommand{\PP}{\mathbb P}
\newcommand{\QQ}{\mathbb Q}

\newcommand{\cB}{\mathcal B}

\newcommand{\sF}{\mathscr F}
\newcommand{\sG}{\mathscr G}
\newcommand{\sH}{\mathscr H}
\newcommand{\sK}{\mathscr K}

\newcommand{\tup}[1]{\langle#1\rangle}

\author{Asaf Karagila}
\thanks{The author was supported by the Royal Society grant no.~NF170989.}
\email[Asaf Karagila]{karagila@math.huji.ac.il}
\urladdr{http://karagila.org}
\address{School of Mathematics,
University of East Anglia.
Norwich, NR4~7TJ, UK
}
\date{\today}
\subjclass[2010]{Primary 03E25; Secondary 03E35}
\keywords{axiom of choice, symmetric extensions, iterations of symmetric extensions, countable union theorem}

\title{The Morris model}

\begin{document}
\begin{abstract}
Douglass B.\ Morris announced in 1970 that it is consistent with $\ZF$ that ``For every $\alpha$, there exists a set $A_\alpha$ which is the countable union of countable sets, and $\mathcal P(A_\alpha)$ can be partitioned into $\aleph_\alpha$ non-empty sets''. The result was never published in a journal (it was proved in full in Morris' dissertation) and seems to have been lost, save a mention in Jech's ``Axiom of Choice''. We provide a proof using modern tools derived from recent work of the author. We also prove a new preservation theorem for general products of symmetric systems, which we use to obtain the consistency of Dependent Choice with the above statement (replacing ``countable union of countable sets'' by ``union of $\kappa$ sets of size $\kappa$'').
\end{abstract}
\maketitle
\section{Introduction}
In 1970 the Notices of the American Mathematical Society published an announcement by Douglass B.\ Morris, a student of Keisler.\footnote{Notices of the American Mathematical Society, \textbf{17} (1970), p.~577.} Morris announced the following consistency result: Every model of $\ZFC$, $M$, can be extended to a model of $\ZF$, $N$, in which for all $\alpha$ there is a set $A_\alpha$ which is a countable union of countable sets, and $\mathcal P(A_\alpha)$ can be mapped onto $\omega_\alpha$. Moreover, this construction does not change cofinalities, and if $V_\alpha^M$ satisfied $\ZF$, then $V_\alpha^N$ also satisfies $\ZF$ and it reflects the above statement. Morris points out that such $N$ cannot be extended to a model of $\ZFC$ without adding ordinals. This theorem is also important as countable union of countable sets might be relatively small, in the sense that they cannot be mapped onto $\omega_2$, but this shows that $\ZF$ alone cannot even prove there is a bound on the power sets of countable unions of countable sets.

The proof appears in full in Morris' thesis \cite{Morris:thesis}, but was never published in a journal. The knowledge of the theorem survived the departure of Morris from the mathematical research through a mention in Jech's ``Axiom of Choice'' \cite{Jech:AC1973}, as Problem~14 in Chapter~5. This problem is marked by two stars indicating that it is a ``difficult (but solved) problem''. No hints as to the way one should arrive to a solution are given.

Morris gives a rough sketch of his argument in his announcement, which to a modern reader is almost unreadable. In a previous version of this manuscript, we speculated as for the content of the proof in Morris' thesis. Since then we managed to come by a copy of the thesis. The proof of Morris is very sophisticated for its time, and we describe it in \S\ref{sec:original}. The original construction satisfies the countable chain condition, and therefore does not change cofinalities. The proof presented in this paper was found independently of Morris' work, and although it is quite similar to the original proof, it requires $\GCH$ in the ground model (or careful bookkeeping and collapsing of some cardinals). It is also easier to generalize the technique used in our proof and obtain better results. In addition there is the benefit of making the proof accessible to modern readers, as the original proof was written in the language of ramified forcing, before the notions of iterations and symmetric extensions were clarified.

The work presented here is a natural development of the work of the author in \cite{Karagila:2018a}, where the author constructs a model of $\ZF$ where Fodor's lemma fails on all regular cardinals. The global construction in that work is presented as an Easton support product of symmetric extensions, or as first performing a preparation class forcing, and then using the machinery of iterations of symmetric extensions with finite support, developed by the author in \cite{Karagila:2016}, and then taking a finite support product of symmetric extensions. In this work the approach using an Easton product argument is doomed to fail, since it hinges on the existence of an outer model of $\ZFC$ with the same ordinals. We hope that this will be a stepping stone that shows the viability of the method of iterations of symmetric extensions, and draw more people to reformulate old and difficult results in this framework, using the guiding principle of iterations: It is sometimes easier to solve your problems one step at a time.

\subsection{In this paper}
We begin by covering some preliminaries about symmetric extensions and cite the necessary theorems about iterations thereof. We then construct a local version of Morris' theorem, before moving on to the amalgamation of the local version into a global statement. We finish by proving a slightly more general preservation theorem for products of symmetric systems, and use it to answer one of Morris' original questions about weak choice principles by showing that similar constructions can satisfy $\DC_{<\kappa}$.

\subsection*{Acknowledgements}
We would like to thank the anonymous referee for their time reading this paper carefully and for their helpful remarks.

\section{Preliminaries}
As this work deals with cardinality without choice, it is a good idea to clarify what we mean by that. Assuming the axiom of choice, every set can be well-ordered and so every set is equipotent to an ordinal, and the least such ordinal is unique. Without choice, however, there might be sets which cannot be well-ordered. We define the cardinal of a set $a$ as either the least ordinal equipotent with $a$, if such ordinal exists, or its Scott cardinal defined by \[|a|=\{b\mid\exists f\colon a\to b\text{ a bijection, and }b\text{ is of minimal rank}\}.\]

We follow the standard forcing terminology for the most part. We say that $\PP$ is a notion of forcing if it is a preordered set with a maximum element called $1$, it is often simpler to assume that $\PP$ is in fact a complete Boolean algebra, this does not change the generality of the definitions and we willingly ignore any unnecessary remarks on these kind of assumptions in favor of readability.

We follow the convention that $q\leq p$ means that $q$ \textit{extends} $p$, or that it is a \textit{stronger condition}. Two conditions $p,q$ are \textit{compatible} if they have a common extension, and we denote that by $p\compatible q$. If $p$ and $q$ are incompatible we write $p\incompatible q$.

For a name $\dot x$, we say that a name \textit{$\dot y$ appears in $\dot x$} if there is some $p$ such that $\tup{p,\dot y}\in\dot x$. We say that a condition $p$ appears in $\dot x$ if the same holds for $p$. 

Finally, if $\{\dot x_i\mid i\in I\}$ is a collection of names, the canonical way to turn that into a name is denoted by $\{\dot x_i\mid i\in I\}^\bullet=\{\tup{1,\dot x_i}\mid i\in I\}$. This notation extends to ordered pairs and sequences in the obvious way. Note that using this notation canonical names for ground model elements have the form $\check x=\{\check y\mid y\in x\}^\bullet$.

\subsection{Symmetric extensions and their iterations}
Let $\PP$ be a forcing and $\pi$ an automorphism of $\PP$. We extend $\pi$ to $\PP$-name by recursion: \[\pi\dot x=\{\tup{\pi p,\pi\dot y}\mid\tup{p,\dot y}\in\dot x\}.\]
\begin{lemma}[The Symmetry Lemma]
Let $\PP$ be a notion of forcing. For every $\pi\in\aut(\PP)$, every condition $p\in\PP$, every $\varphi$ and every $\PP$-name $\dot x$, \[p\forces\varphi(\dot x)\iff\pi p\forces\varphi(\pi\dot x).\qed\]
\end{lemma}

Suppose that $\sG$ is a group, we say that $\sF$ is a \textit{normal filter of subgroups} if it is a filter on the lattice of subgroups of $\sG$ which is closed under conjugations. Namely, it is a non-empty collection of subgroups which is closed under supergroups and intersections, and whenever $H\in\sF$ and $\pi\in\sG$, $\pi H\pi^{-1}\in\sF$.

\begin{definition}We say that $\tup{\PP,\sG,\sF}$ is a \textit{symmetric system} if $\PP$ is a notion of forcing, $\sG\subseteq\aut(\PP)$ and $\sF$ is a normal filter of subgroups over $\sG$.\footnote{It is enough that $\sF$ is a normal filter base, and we will often define a filter base and ignore the rest of the filter it generates.}
\end{definition}

If $\sG$ is such that for all $p$ and $q$ in $\PP$ there is some $\pi\in\sG$ such that $\pi p\compatible q$, we say that \textit{$\sG$ witnesses the homogeneity of $\PP$}, or that $\tup{\PP,\sG,\sF}$ is a \textit{homogeneous system}. We say that a system is \textit{strongly homogeneous} if for every $H\in\sF$ there is some condition $p$ such that $\pi p=p$ for all $\pi\in H$ and $H$ witnesses the homogeneity of the cone below $p$.

Fix a symmetric system $\tup{\PP,\sG,\sF}$. For a name $\dot x$, $\sym_\sG(\dot x)$ is the subgroup $\{\pi\in\sG\mid\pi\dot x=\dot x\}$. We say that $\dot x$ is $\sF$-symmetric if $\sym_\sG(\dot x)\in\sF$. If this condition holds hereditarily, we say that $\dot x$ is a \textit{hereditarily $\sF$-symmetric} name, and we denote by $\HS_\sF$ the class of all hereditarily $\sF$-symmetric names.

\begin{theorem}
Suppose that $G$ is a $V$-generic filter for $\PP$, and let $M$ be the class $\HS_\sF^G=\{\dot x^G\mid\dot x\in\HS_\sF\}$ in $V[G]$. Then $M$ is a transitive class satisfying $\ZF$ and $V\subseteq M\subseteq V[G]$.\qed
\end{theorem}
The class $M$ in the theorem is often referred to as a \textit{symmetric extension}. We can define the symmetric forcing relation $\forces^\HS$ as the relativization of $\forces$ to the class $\HS_\sF$, and we can prove that $\HS_\sF^G\models\varphi(\dot x^G)$ if and only if $\exists p\in G:p\forces^\HS\varphi(\dot x)$. Moreover, the symmetry lemma holds for $\forces^\HS$ whenever the permutations come from $\sG$. Note that by transitivity of $M$, for bounded formulas, $\forces$ and $\forces^\HS$ are the same.

If the symmetric system is clear from context, we omit $\sG$ and $\sF$ from the subscripts and the terminology.

We can iterate symmetric extensions when using finite support, this theory was developed by the author in \cite{Karagila:2016}. The theory itself is fairly comprehensive, but we will only use a small fraction of it in this paper. Specifically, we will use the preservation theorems (Theorems~9.2,~9.4 in \cite{Karagila:2016}).
\begin{theorem}
Let $\tup{\QQ_\alpha,\sG_\alpha,\sF_\alpha\mid\alpha\in\Ord}$ be a finite support iteration of symmetric extensions such that for all $\alpha$, $\forces_\alpha\tup{\QQ_\alpha,\sG_\alpha,\sF_\alpha}$ is a homogeneous system. Assume that for any $\eta$ there is some $\alpha^*$, such that for all $\alpha\geq\alpha^*$, the $\alpha$th symmetric extension does not add new sets of rank $\eta$. Then no sets of rank $\leq\eta$ are added by limit steps either. In particular the end model satisfies $\ZF$.
\end{theorem}
In other words, assuming each iterand is homogeneous, if after some stage $\alpha^*$ we no longer add sets to $V_\eta$ at each successor step, then the same holds for limit stages. If this is true for all $\eta$, $\ZF$ is preserved. This is important, since finite support iterations tend to add Cohen reals and collapse cardinals when the forcings are not c.c.c.\ themselves. Which would be disastrous for us, and in fact would require us to check the axioms of $\ZF$ hold in the resulting model by hand (see proofs by Gitik in \cite{Gitik:1980} and Fernengel--Koepke in \cite{Fernengel-Koepke:2016} for example).

\section{Local version}\label{sect:local}
Let $\kappa$ be a regular cardinal, and without loss of generality $\kappa^{<\kappa}=\kappa$. We first construct a symmetric extension in which there is a set which is a countable union of countable sets, then we construct a symmetric extension of that model in which no sets of ordinals are added, where the aforementioned set's power set can be mapped onto $\kappa$.

\subsection{The first symmetric extension}
Let $\PP$ be $\Add(\kappa,\omega\times\omega\times\kappa)$, so a condition is a function $p\colon\omega\times\omega\times\kappa\times\kappa\to 2$ such that $|\dom p|<\kappa$.

Let us define some canonical names for our objects of interest:
\begin{enumerate}
\item $\dot x_{n,m,\alpha}=\{\tup{p,\check\beta}\mid p(n,m,\alpha,\beta)=1\}$,
\item $\dot a_{n,m} = \{\dot x_{n,m,\alpha}\mid\alpha<\kappa\}^\bullet$,
\item $\dot A_n = \{\dot a_{n,m}\mid m<\omega\}^\bullet$,
\item $\vec A=\tup{\dot A_n\mid n<\omega}^\bullet$.
\end{enumerate}

The goal is to have all of these names symmetric, and while the canonical names for the enumeration of each $\dot A_n$ should be symmetric, we do not want them to be uniformly symmetric, since our goal is to have the union of these $\dot A_n$'s our set whose power set will (eventually) be mapped onto $\kappa$.

The automorphism group we use is that of permutations $\pi$ of $\omega\times\omega\times\kappa$, such that if $\pi(n,m,\alpha)=(n',m',\alpha')$, then $n=n'$ and $\pi``\{n\}\times\{m\}\times\kappa=\{n\}\times\{m'\}\times\kappa$. Namely, we first apply a permutation of $\omega$ to the second coordinate, and then separately for each $m$, $\pi(n,m,\cdot)$ is a permutation of $\kappa$. In group theoretic terms, $\sG$ would be the wreath product, or $\sG=\{\id\}\wr S_\omega\wr S_\kappa$. Of course, the action on $\PP$ is standard and given by \[\pi p(\pi(n,m,\alpha),\beta)=p(n,m,\alpha,\beta).\]

Let us denote by $\pi_n$ and $\pi_{n,m}$ the permutations which are obtained by fixing $n$ and both $n,m$ respectively. We shall denote by $\pi_n^*$ the permutation of $\omega$ given by $\pi_n^*(m)=m'$ if and only if $\pi(n,m,0)=(n,m',\alpha)$ for some $\alpha$.

\begin{proposition}
Let $\pi\in\sG$, then
\begin{enumerate}
\item $\pi\dot x_{n,m,\alpha}=\dot x_{n,\pi_n(m,\alpha)}$; 
\item $\pi\dot a_{n,m}=\dot a_{n,\pi_n^* m}$; 
\item $\pi\dot A_n=\dot A_n$; and
\item $\pi\vec A=\vec A$.\qed
\end{enumerate}
\end{proposition}

For $E\subseteq\omega\times\omega\times\kappa$, denote by $\fix(E)$ the group \[\{\pi\in\sG\mid\forall\tup{n,m,\alpha}\in E:\pi^*_n=\id\text{ and }\pi\restriction E=\id\}.\] Namely, $\fix(E)$ is the group of permutations in $\sG$ which do not move the $A_n$'s which are mentioned in $E$, as well as the coordinates which are in $E$ otherwise. We let $\sF$ be the normal filter of subgroups generated by $\fix(E)$ for a finite $E$. We say that $E$ is a \textit{support} for a name $\dot x$ if $\fix(E)\subseteq\sym(\dot x)$.

\begin{proposition}
All the above names are in $\HS$.
\end{proposition}
\begin{proof}
Taking $\{\tup{n,m,\alpha}\}$ is a support for $\dot x_{n,m,\alpha}$ and $\dot a_{n,m}$. Any permutation preserves $\dot A_n$ for any $n$, and therefore $\vec A$.
\end{proof}
The following proposition shows that $\bigcup \dot A_n$ is a name of an uncountable set which is countable union of countable sets, it is also a good exercise in symmetry arguments.
\begin{proposition}\label{prop:countable-union-is-uncountable}
$\forces^\HS\forall n<\omega,|\dot A_n|=\aleph_0$ and $|\bigcup\dot A_n|>\aleph_0$.
\end{proposition}
\begin{proof}
The first part is an easy consequence of the above proposition. The name $\tup{\dot a_{n,m}\mid m<\omega}^\bullet$ of the canonical enumeration is fixed pointwise once $(n,m,\alpha)\in E$ for some $m,\alpha$. It is enough to prove that there is no uniform enumeration of all the $\dot A_n$'s which is in $\HS$ in order to show that the union is uncountable.

Indeed, suppose that $\dot f\in\HS$ was such that $p\forces``\forall n,m<\omega\ \dot f(\check n,\check m)\in\dot A_n$ and $\dot f$ is injective''. Let $E$ be a support for $\dot f$, and let $n$ be such that $\tup{n,m,\alpha}\notin E$ for all $m,\alpha$. By extending $p$ if necessary  we can assume that for some $m<\omega$, $p\forces\dot f(\check n,\check m)=\dot a_{n,k}$ for some $k$. Let $k'\neq k$, we can find two permutations of $\kappa$, such that for $\pi\in\sG$ for which $\pi^*_n$ is the $2$-cycle switching $k$ and $k'$, and our two permutations of $\kappa$ are $\pi_{n,k}$ and $\pi_{n,k'}$ respectively, such that $\pi p$ is compatible with $p$. For example, let $\alpha<\kappa$ be larger than $\sup\{\beta\mid\tup{n,k,\beta}\in\dom p\}$ and let $\pi_{n,k}$ be the permutation switching the blocks $[0,\alpha)$ and $[\alpha,\alpha+\alpha)$. In fact, by choosing $\alpha$ to be large enough, we can even assume that $\pi_{n,k}=\pi_{n,k'}$. Anywhere else, define $\pi$ as the identity.

We now that $\pi p\forces\pi\dot f(\check n,\check m)=\pi\dot a_{n,k}=\dot a_{n,k'}$, but since $\pi p$ and $p$ are compatible this means that $p\nforces``\dot f$ is injective'', Therefore the union is uncountable.
\end{proof}
Note that the proof in fact shows that any countable subset of the union is in fact a subset of finitely many $A_n$'s.

\subsection{The second symmetric extension}
Let $G$ be a $V$-generic filter for $\PP$, and let $M$ be the symmetric extension above. We omit the dots of the names we defined above to denote their interpretation in $M$. Namely, $A_n=\dot A_n^G$ etc.

\begin{definition}
Let $T$ be the choice tree from $\vec A$. Namely, $T=\bigcup_{n<\omega}\prod_{k<n}A_k$, ordered by inclusion. 
\end{definition}
By the arguments similar to \autoref{prop:countable-union-is-uncountable}, it is not hard to see that $T$ has no branches in $M$.

For $s\in\omega^{<\omega}$ define $\dot t_s$ as $\tup{\dot a_{i,s(i)}\mid i\in\dom s}^\bullet$. Given a $\PP$-name $\dot t$ and a condition $p\in\PP$ forcing $\dot t\in\dot T$, by extending $p$ finitely many times if necessary, we can decide the values of $t$ and thus ensure that it has the form $\dot t_s$ for some $s\in\omega^{<\omega}$. Note that $\dot t_{s\restriction k}=\dot t_s\restriction k$ as well.

We would like to add branches to $T$, which are subsets of $A$, so that the new subsets can be partitioned into $\kappa$ parts. Of course, the goal is to do so without adding any subsets to the original ground model. This is important for applying the preservation theorem in the global case. It also shows that under $\GCH$ no cardinals are collapsed. 
\medskip

Let us digress from the proof to discuss the motivation behind the definitions we are about to give. The main tool for achieving our goal would be to define a homogeneous symmetric system in $M$, such that conditions can be also moved by applying permutations of $\PP$. This would mean that if $\dot a$ is a symmetric name for a subset of $V$, we can first restrict the relevant conditions of the symmetric system in $M$ by means of homogeneity, and then look at it as a name in the iteration given by $\PP$ and our forcing, and we will apply permutations of $\PP$ to ensure that we only need to extend conditions in $\PP$ to decide $\check b\in\dot a$. In turn it means that in $M$, $\dot a$ was a name for a ground model (read: $M$-) set.

Naturally, we want to force with some copies of $T$, say $\kappa\times\omega$ of them, and we will allow permutations of each $\omega$-block as to ``fuzzy things out''. The natural thing is to use finite support products. However, when we come to apply the clever trick of going back to $\PP$ and using its automorphisms to make two conditions compatible we would run into a problem. Given a condition with just two points $t_0$ and $t_1$, if we want to make it compatible with a condition having $s_0$ and $s_1$ in its nontrivial part, and the $n$th level of both $s_0$ and $s_1$ chose the same element of $A_{n-1}$, while $t_0$ and $t_1$ chose different elements, then we are stuck. We cannot move the element chosen by $t_0$ to the one chosen by $s_0$, because then we cannot move the corresponding choice at $t_1$. Remember, we are applying automorphisms of $\PP$, so we are moving the underlying sets, the $A_n$'s and in a rather global way.

The solution for this is to have ``injective conditions'', these would add branches which are completely disjoint from one another. But this is a problem, since that would add a collapsing function making $\kappa$ countable. This happens because we can ask what is the least $\alpha$ that has the $m$th element of $A_0$ in a set which is in the $\alpha$th block of the sequence. 

There are two apparent solution. The first, note that there are sets in $M$ which are Dedekind-finite and can be mapped onto $\kappa$ (e.g.\ each $a_{n,m}$), and we can simply add ``that many pairwise disjoint branches''. If we are careful, we will not destroy the Dedekind-finiteness, and thus add only a surjection, rather than an injection. This seems a bit ad-hoc and incidental. Instead, we will take a different approach, where we keep track of the disjointness of our conditions. This will allow us to appeal to the original argument, but it will be sufficiently non-descript that no subsets of $V$ will be added, and in particular no sets of ordinals.

\begin{definition}
Suppose that $t_0,\ldots,t_{n-1}$ are finite branches in $T$. We say that the sequence is \textit{$m$-injective} if for all $n>m$, $t_i(n)\neq t_j(n)$ for $i<j<n$. 
\end{definition}

For a sequence $\vec t\in T^{\kappa\times\omega}$, let $\supp(\vec t)$ denote the set of pairs $\tup{\alpha,n}$ such that the $\tup{\alpha,n}$th branch in $\vec t$ is nontrivial. We use the standard terminology of \textit{support} here. While this term is quite extensively used in different contexts in this paper, these uses are deeply connected to one another. For readability, we will write $t$ etc., to denote $\vec t$.
\medskip

Define $\QQ$ to be the forcing with conditions $\tup{t,f_t}$ such that $ t\in T^{\kappa\times\omega}$ is a finite sequence of trees, and $f_t\colon\mathcal P(\supp( t))\to\omega$ such that $f_t$ is $\subseteq$-non decreasing and $ t\restriction a$ is $f_t(a)$-injective, this $f_t$ is the \textit{disjointing function of $ t$}. The \textit{support} of a condition is the support of its $ t$-part.

If $E\subseteq\kappa\times\omega$, write $\tup{t,f_t}\restriction E$ to denote the condition $\tup{t\restriction E,f_t\restriction(E\cap\supp( t)}$. In particular, the support of $\tup{t,f_t}\restriction E$ is a subset of $E$.

For $\tup{t,f_t}$ and $\tup{s,f_s}$ in $\QQ$ we say that $\tup{t,f_t}\leq\tup{s,f_s}$ if the following conditions hold:
\begin{enumerate}
\item $\supp( s)\subseteq\supp( t)$ and for all $\tup{\alpha,n}\in\supp( s)$, $s_{\tup{\alpha,n}}\subseteq t_{\tup{\alpha,n}}$.
\item For all $a\subseteq\supp( s)$, $f_t(a)\leq f_s(a)$.
\end{enumerate}

Clearly, if $f$ and $g$ are two functions such that $\tup{t,f},\tup{t,g}\in\QQ$, then $\tup{t,f}$ and $\tup{t,g}$ are compatible.

Every condition in $\QQ$ has a canonical name, where $ t$ is composed of names of the form $\dot t_s$ for some $s\in\omega^{<\omega}$ along with their index, and $\dot f_t$ has a canonical name given by the behavior of the function $f_t$ as a function from a power set of the suitable finite subset of $\kappa\times\omega$ into $\omega$.

We define $\sH$ to be the group of permutations $\pi$ of $\kappa\times\omega$ such that for all $\alpha$, there is some $\pi_\alpha$ which is a permutation of $\omega$ and $\pi(\alpha,n)=\tup{\alpha,\pi_\alpha(n)}$. These act on $\QQ$ in the standard way, applying $\pi$ to $ t$ simply acts on the support, and it is not hard to verify that this action extends to $f_t$ in the appropriate manner.

Finally, as would be expected, our filter of subgroups is defined as generated by $\fix(E)$ for $E\in[\kappa\times\omega]^{<\omega}$, where $\fix(E)=\{\pi\in\sG\mid\pi\restriction E=\id\}$. As with our standard terminology $E$ is a support for a name $\dot x$ when $\fix(E)\subseteq\sym(\dot x)$. 

For $\tup{\alpha,n}\in\kappa\times\omega$, we define the name for the set added by the $\tup{\alpha,n}$th branch: \[\dot B_{\alpha,n}=\{\tup{\tup{t,f_t},\check a}\mid\supp( t)=\{\tup{\alpha,n}\}\land\exists m: t(\alpha,n,m)=a\}.\]
Clearly, for $\pi\in\sH$, $\pi\dot B_{\alpha,n}=\dot B_{\pi(\alpha,n)}$. It is also clear that each $\dot B_{\alpha,n}$ is in $\HS$, since $\{\tup{\alpha,n}\}$ is a support witnessing that. It is also clear that for each $\alpha$, $\dot\cB_\alpha=\{\dot B_{\alpha,n}\mid n<\omega\}^\bullet$ and $\tup{\dot\cB_\alpha\mid\alpha<\kappa}^\bullet$ are preserved by all the automorphisms in $\sG$, and are therefore in $\HS$ as well. So $A$, which remains a countable union of countable sets in the symmetric extension, will admit a surjection from $\mathcal P(A)$ onto $\kappa$, as wanted.

\begin{proposition}
Suppose that $\dot x\in\HS$ such that for some $a$, $\forces\dot x\subseteq\check a$. Let $E$ be a support for $\dot x$, if $\tup{t,f_t}$ decides $\check b\in\dot x$, then $\tup{t,f_t}\restriction E$ already decides the same.
\end{proposition}
\begin{proof}
Suppose that $\tup{t',f_{t'}}\leq\tup{t,f_t}\restriction E$, then we can find a permutation in $\fix(E)$ which moves any coordinate in $\supp t\setminus E$ to a coordinate outside $\supp( t)$. Let $\pi$ be such automorphism of $\QQ$, then $\pi\dot x=\dot x$ and so $\pi\tup{t',f_{t'}}$ is compatible with $\tup{t,f_t}$. If $\tup{t,f_t}\forces\check b\in\dot x$, then no compatible condition can force otherwise. In particular, $\pi\tup{t',f_{t'}}\nforces\check b=\pi\check b\notin\pi\dot x=\dot x$, so $\tup{t',f_{t'}}$ cannot force $\check b\notin\dot x$. Therefore $\tup{t,f_t}\restriction E$ already forced $\check b\in\dot x$. The argument for $\check b\notin\dot x$ is similar, and therefore the conclusion holds.
\end{proof}
\begin{corollary}\label{cor:gnd-model-names}
Suppose that $\dot x\in\HS$ with support $E$ such that for some $a$, $\forces\dot x\subseteq\check a$. Define $\dot x_*=\{\tup{\tup{t,f_t}\restriction E,\check b}\mid\tup{t,f_t}\forces\check b\in\dot x\}$, then $\dot x_*\in\HS$ and $\forces\dot x=\dot x_*$.\qed
\end{corollary}

Say that a finite $E\subseteq\omega$ \textit{captures} a condition $\tup{t,f_t}$ if for all $\tup{\alpha,n}\in\supp( t)$, $\dom t_{\alpha,n}\subseteq E$ and $\rng f_t\subseteq E$. We say that two conditions $\tup{s,f_s}$ and $\tup{t,f_t}$ are \textit{compatible on $E$}, if for any $\tup{\alpha,i}\in\supp( t)\cap\supp( s)$ and $n\in E$, $s_{\alpha,i}\restriction n=t_{\alpha,i}\restriction n$.

\begin{lemma}\label{lem:upwards-homogeneity}
Let $\tup{s,f_s}$ be a condition in $\QQ$ and let $n<\omega$ be such that $n$ captures $\tup{s,f_s}$. For any $q,q'\leq\tup{s,f_s}$ with $\supp(q)=\supp(q')=\supp( t)$, if $q$ and $q'$ are compatible on $n$, then there is an automorphism of $\PP$ in $\fix_\sG(n\times\{0\}\times\{0\})$ which does not move the canonical name of $\tup{t,f_t}$, and $\forces_\PP\dot q\compatible_\QQ\pi\dot q'$, where $\dot q$ and $\dot q'$ are the canonical names for $q$ and $q'$ respectively.
\end{lemma}
This statement is complicated to state and explain in words, and we encourage the reader to get a pen and a piece of paper and try to draw the scenario in the assumptions of the lemma regarding $\tup{s,f_s}$ and $q,q'$. The rest of the paper can wait. It might also be fruitful to skip the proof of this lemma on first reading.
\begin{proof}
Let $q=\tup{t,f_t}$ and $q'=\tup{t',f_{t'}}$. Without loss of generality we may assume that for all $\tup{\alpha,i}\in\supp(q)=\supp(q')$, $\dom t_{\alpha,i}=\dom t'_{\alpha,i}$.

Fix some $E$ which captures both $q$ and $q'$. For simplicity, assume that $E=n+1$. The general proof will then be the result of a recursive application of the simplified proof.

Let $ c$ denote the sequence $c_{\alpha,i}$ such that for $\tup{\alpha,i}\in\supp t$, $t_{\alpha,i}(n)=c_{\alpha,i}$. Similarly, define $ c'$. Since $n$ captures $\tup{t,f_t}$, it has to be the case that \[| c|=|\{\tup{\alpha,i}\in\supp( t)\mid n\in\dom t_{\alpha,i}\}|=| c'|.\]

Let $m(\alpha,i)$ be such that $\dot a_{k-1,m(\alpha,i)}$ is the canonical name for $c_{\alpha,i}$, and similarly define $m'(\alpha,i)$. Now consider $\sigma$ to be the permutation of $\omega$ which maps $m'(\alpha,i)$ to $m(\alpha,i)$. And define $\pi\in\sG$ to be the permutation for which $\pi_n^*=\sigma$, and $\pi_k$ to be the identity for $k\neq n$.\footnote{Note that we do not restrict $\pi_{n,m}$ for any $m$, as it has no effects on these names.}

It is easy to see why such $\pi$ preserves the canonical name of $\tup{s,f_s}$, and why $\pi\dot q'$ is compatible with $\pi q$. Indeed, $\pi\dot q'=\dot q$. We can now apply this argument recursively to obtain the wanted consequence without assuming $E=n+1$. Finally, removing the first assumption that $\dom t_{\alpha,i}=\dom t'_{\alpha,i}$ for all $\tup{\alpha,i}\in\supp(q)$, we simply obtain $\sigma$ for the common parts. The rest will not interfere with compatibility.
\end{proof}

\begin{proposition}
Suppose that $\dot x$ is a name in $\HS$ such that $t\forces\dot x\subseteq\check V$. Then there is some $t'$ which extends $t$, such that for some $u\in V$, $t'\forces\dot x=\check u$.
\end{proposition}
\begin{proof}
By \autoref{cor:gnd-model-names} we may assume that $\dot x=\dot x_*$, in the notation of that corollary, in particular every name which appears in $\dot x$ has the form $\check y$ for some $y\in V$. Let $[\dot x]$ denote the $\PP\ast\QQ$-name given by $\{\tup{\tup{p,\dot q},\check y}\mid p\forces\tup{\dot q,\check y}\in\dot x\}$.\footnote{For $\check y$ we ignore the distinction of what forcing is being used, if $\varnothing$ is assumed to be the maximum condition of all notions, then it is truly the same object when $y\in V$.}

Namely, $[\dot x]$ is the translation of $\dot x$ from a $\QQ$-name in the symmetric extension given by $\PP$, to a $\PP\ast\dot\QQ$-name. We may assume that all the finite branches in $\dot q$ have canonical form, and that the disjointing function also has a canonical name defined similar to the names $\dot t_s$. Let $E$ be a support for $[\dot x]_\PP$, the $\PP$-name for $\dot x$ in $\HS_\sF$.

Suppose that $p\forces_\PP\dot q\forces_{\dot\QQ}\check y\in[\dot x]_\PP$, which is the same as saying that $\tup{p,\dot q}\forces\check y\in[\dot x]$. By applying \autoref{lem:upwards-homogeneity}, we get that we may restrict $\dot q$ such that for all $\dot t_s$ which are in the support of $\dot q$, we can restrict $s$ to $\max E$. Namely, any condition which is compatible with $\dot q$ on $E$ must force the same.

This is enough, modulo one minor problem, that when we apply $\pi$ obtained by \autoref{lem:upwards-homogeneity}, it might be that $\pi p$ is incompatible with $p$. To remedy that, we simply add the following condition to $\pi$: for all $m<\omega$, $\pi_{n,m}$ moves the domain of $p(n,m,\cdot)$ to a disjoint interval, similar to the way we defined this in the end of the proof of \autoref{prop:countable-union-is-uncountable}. This does not affect the consequence of the lemma, and ensures that $p$ is compatible with $\pi p$.

Suppose now that $q=\tup{t,f_t}$ is a condition such that $\dom t_{\alpha,i}=\max E$ for all $\tup{\alpha,i}\in\supp( t)$, then we can define the $\PP$-name $\dot u_q=\{\tup{p,\check y}\mid p\forces^\HS_\PP\dot q\forces_{\dot\QQ}\check y\in[\dot x]_\PP\}$. This name is symmetric since $E$ is a support for it. And in $M$, $q\forces\dot x=\check u_q$.
\end{proof}
\begin{corollary}\label{cor:dist}
If $\alpha$ is such $|V_\alpha|<\kappa$, then no new sets are added of rank $\alpha$ when taking a symmetric extension with $\tup{\QQ,\sH,\sK}$ over $M$.\qed
\end{corollary}
\section{Morris' theorem}
\begin{theorem}[Morris' theorem]
It is consistent that for every $\alpha$ there exists a set $A_\alpha$ which is the countable union of countable sets, and $\mathcal P(A_\alpha)$ can be mapped onto $\alpha$.
\end{theorem}
\begin{proof}

Assume that $V\models\ZFC+\GCH$. For every $\alpha$, let $\tup{\QQ_{\alpha,0},\sG_\alpha,\sF_\alpha}$ denote the first symmetric system described in the previous section as defined in $V$ where $\kappa=\omega_{\alpha+1}$, and let $\tup{\dot\QQ_{\alpha,1},\dot\sH_\alpha,\dot\sK_\alpha}$ denote the symmetric system described in the second section (here $\QQ_{\alpha,1}$ is composed of the canonical names of the conditions, etc.), as defined in the symmetric extension of $V$ given by the first step. We now define a symmetric iteration over the class of ordinals, where at each step we force with $\QQ_{\alpha,0}$ and then with $\QQ_{\alpha,1}$. We denote by $\QQ_\alpha$ the iteration $\QQ_{\alpha,0}\ast\dot\QQ_{\alpha,1}$.

Let $\PP_\alpha$ the iteration of the first $\alpha$ steps, and $\PP=\PP_\Ord$. 

\begin{lemma}
For all $\alpha$, $\forces^\IS_\alpha\QQ_\alpha$ adds no sets of rank $\eta$ such that $|V_\eta|<\aleph_{\alpha+1}$.
\end{lemma}
\begin{proof}
Since $\QQ_\alpha$ is in fact defined in $V$, the iteration is in fact a product of two-step iterations. This means that we can change the order in which we add the generics, in particular, by \autoref{cor:dist}, $\QQ_\alpha$ does not add sets of rank $\eta$ to the ground model. Therefore the conclusion follows.
\end{proof}

Since each iterand is weakly homogeneous, we can apply the preservation theorem and obtain a model in which for each $\alpha$ there is a set $A_\alpha$ which is the countable union of countable sets, and $\mathcal P(A_\alpha)$ can be mapped onto $\omega_\alpha$ (in fact, onto $\omega_{\alpha+1}$).
\end{proof}

We also derive the corollary given by Morris himself in his announcement.
\begin{corollary}
Given a model of $\ZFC$, it has an extension satisfying $\ZF$ which cannot be extended further to a model of $\ZFC$ without adding ordinals.
\end{corollary}
\begin{proof}
Given a model of $\ZFC$, extend it to a model as in Morris' theorem, $M$.\footnote{If $\GCH$ does not hold in $M$ the extension might collapse cardinals. This can be overcome by introducing ``gaps'' between nontrivial iterands so that enough cardinals survive the construction.} If $M\subseteq N$ and $N\models\ZFC$, then each $A_\alpha$ is countable, and therefore $\mathcal P(A_\alpha)$ is of size $2^{\aleph_0}$. In particular, all the ordinals of $M$ have size at most $2^{\aleph_0}$.
\end{proof}

It is worth noting there are currently three known models which satisfy this corollary: Gitik's model from \cite{Gitik:1980}, where all limit ordinals have countable cofinality and very large cardinals are necessary; Fernengel--Koepke models from \cite{Fernengel-Koepke:2016}, where no large cardinals are used, and the Singular Cardinals Hypothesis is violated in a certain way that ensures that extending the model of a model of $\ZFC$ will collapse all cardinals to become countable; and the Morris model, as described here and in the original proof.

Seemingly interesting, the construction of the Morris' model seem to mainly violate the power set axiom. Indeed all cardinals are collapsed by the very nature of the finite support product of non-c.c.c.\ partial orders. However, this might be salvageable by introducing a sort of mixed support where we take an Easton product of symmetric extensions first, and a finite support product of the second-iterands in that model.\footnote{This is similar to the use of symmetric iterations in \cite{Karagila:2018a}.} In such situation, the power set axiom will necessarily be violated.

\subsection{Morris' original proof}\label{sec:original}
Morris' original proof is a marvel of creation, considering the fact it was written in 1970, before unramified forcing was popularized by Shoenfield, and before treatises on symmetric extensions or on iterated forcing were written. Let alone any framework for iterating symmetric extensions.

Morris' proof goes along the following lines: add Cohen reals to create a set $C_0$ which is a countable union of countable sets of Cohen reals, and add subsets to its power set as described above; next add Cohen subsets to $C_0$ in the same fashion, to create $C_1$ which is a countable union of countable sets of Cohen subsets of $C_0$, and add subsets to its power set (this time mapping it onto $\omega_1$); and so on. At limit steps branches are introduced guided by ground model functions from $\alpha$ to $\omega$ and partitioned based on equality modulo a canonically chosen ultrafilter on $\alpha$ (in the case of $L$, the least constructible one).

This definition guarantees that the result is a c.c.c.\ forcing, which therefore does not change cofinalities. It is, however, externally equivalent to adding a proper class of Cohen reals, which therefore violates the axiom of power set. By defining the intermediate model carefully, though, the only subsets of $\omega$ which are added are those added in the first step.

It might be worth formulating Morris' work into the framework developed by the author in his Ph.D.\ thesis, and it seems that Morris also produced the first example of a model of $\ZF$ where Kinna--Wagner principles fail completely, or so we conjecture based on the construction (see \cite[\S10]{Karagila:2016} or \cite[\S5]{Karagila:2018b} for a full discussion on Kinna--Wagner principles and related results).

This is in contrast to the proof we present here, where the key idea is to create a localized failure and create a global failure by gluing together the local versions. While this new proof is perhaps simpler and more straightforward, it does come at the cost of requiring $\GCH$ (or allowing cardinals to collapse). Not a terrible price to pay, overall, but it is a technicality which is necessary to deal with. The other obvious advantage of a simpler proof is given in the next section: It is easier to generalize.

\section{Accommodating Dependent Choice}
The proof we gave to Morris' theorem relies heavily on the preservation theorem, which is given in the context of iterations with finite support. This makes the preservation of Dependent Choice quite impossible to achieve in nontrivial situations. The main difference of our construction from Morris' original proof is that ours is a product of localized versions, whereas Morris does what seems to be a proper iteration. If one wants to accommodate $\DC_{<\kappa}$, then one needs to move from finite support to $\kappa$-support iterations. Whereas a theory of iterations of symmetric extensions which are not finitely supported is nowhere near the horizon, we can prove a preservation theorem which captures this instance (and actually much more).

\begin{theorem}\label{thm:preservation}
Suppose that $\tup{\QQ_\alpha,\sG_\alpha,\sF_\alpha\mid\alpha<\delta}$ are strongly homogeneous symmetric systems such that $\QQ_\alpha$ is $\kappa$-closed and $\sF_\alpha$ is $\kappa$-complete. Assume that for all $\alpha<\delta$, in the symmetric extension given by the $\kappa$-support product $\prod_{\xi<\alpha}\QQ_\xi$, the symmetric system $\tup{\QQ_\alpha,\sG_\alpha,\sF_\alpha}$ does not add any sets of rank $<\eta$. Then the $\kappa$-support product of $\prod_{\alpha<\delta}\QQ_\alpha$ does not add any sets of rank $<\eta$.
\end{theorem}
The proof is quite similar to the proof of Theorem~9.2 in \cite{Karagila:2016}. So we only sketch the main ideas behind it.
\begin{proof}[Sketch of Proof]
We prove this by induction on $\delta$. For $\delta=0$ and successor steps this is trivial. So we can assume that $\delta$ is a limit ordinal.

Suppose that $\dot x$ is a name of a new subset of minimal rank. In particular we can assume that if $\dot y$ appears in $\dot x$, then $\dot y$ has the form $\check u$ for some $u$ in the ground model.

If $\dot x\in\HS$, then there is a sequence of groups $\tup{H_\alpha\mid\alpha<\delta}$ such that $H_\alpha\in\sF_\alpha$ and $|\{\alpha\mid H_\alpha\neq\sG_\alpha\}|<\kappa$, denote this set as $S$. By homogeneity it follows that if $p\forces\check u\in\dot x$, then $p\restriction S\forces\check u\in\dot x$, and similarly for $p\forces\check u\notin\dot x$. 

We can assume, therefore, that $\dot x$ is in fact a $(\prod_{\alpha\in S}\QQ_\alpha)$-name, note that this product is a full support product. If $\cf(\delta)\geq\kappa$, then $S$ is bounded in $\delta$, so by the induction hypothesis it is a ground model set. So we may assume that $\cf(\delta)<\kappa$. However, since we only care about the coordinates in $S$, this is really a product of $\otp(S)$ forcings, which again the induction hypothesis deals with whenever $\delta\geq\kappa$. So we may assume that $S=\delta<\kappa$.

Finally, for each $\alpha<\delta$ find a maximal antichain $D_\alpha$ of conditions which are fixed pointwise by $H_\alpha$. By strong homogeneity, we can assume that $H_\alpha$ witnesses the homogeneity of the cone below each condition in $D_\alpha$. It follows, therefore, that any condition in $\prod_{\alpha<\delta} D_\alpha$ decides all the statements of the form $\check u\in\dot x$. Since the product is a full support product, this produces a maximal antichain in $\prod_{\alpha<\delta}\QQ_\alpha$. Therefore $\dot x$ is equivalent to a ground model name.
\end{proof}
\begin{remark}
\begin{enumerate}
\item It might feel like the proof shows that no new sets are added, since the assumption on the rank was not used. However it was used in the successor step. Any new sets added by successor steps, where we actually force over the intermediate model (rather than taking a limit of some sort). 
\item Much like the case of the preservation theorem in \cite{Karagila:2016}, this too can be replaced by any kind of hierarchy which is sufficiently nice.
\item It is not hard to see that the resulting model is $\kappa$-closed in the full generic extension. Therefore by \cite{Karagila:2018c} $\DC_{<\kappa}$ holds as well.
\item We can replace each single symmetric system by a finite iteration of symmetric systems, provided that the conditions about closure and not adding sets still hold.
\end{enumerate}
\end{remark}
\begin{corollary}
If $\tup{\QQ_\alpha,\sG_\alpha,\sF_\alpha\mid\alpha\in\Ord}$ is such that for all $\eta$ there is some $\alpha^*$ such that for all $\alpha>\alpha^*$, in the symmetric extension given by the $\kappa$-support product of $\prod_{\xi<\alpha}\QQ_\xi$, taking the symmetric extension given by $\tup{\QQ_\alpha,\sG_\alpha,\sF_\alpha}$ does not add sets of rank $\eta$, then the intermediate model satisfies $\ZF+\DC_{<\kappa}$.
\end{corollary}
The proof here is exactly the same proof as in \cite{Karagila:2016}. The idea behind it is that the power sets are fixed once we stop adding sets of sufficiently high rank, and Replacement holds for a similar reason.

\begin{theorem}[The $\kappa$-Morris model]
Suppose that $V\models\GCH$, then for every regular $\kappa$, there is an extension $M$ of $V$ which satisfies the following properties:
\begin{enumerate}
\item $\cf(\alpha)^M=\cf(\alpha)^V$ for all $\alpha$.
\item $\ZF+\DC_{<\kappa}$ holds.
\item For every $\alpha$, there is a set $A_\alpha$ which is the union of $\kappa$ sets of size $\kappa$, and $\mathcal P(A_\alpha)$ can be mapped onto $\omega_\alpha$.
\item There is no extension of $M$ to a model of $\DC_\kappa$ with the same ordinals.
\end{enumerate}
\end{theorem}
The proof does by replacing $\omega$ by $\kappa$ in the construction given in the previous section, noting that all the conditions for the preservation theorems are satisfied.

This model shows an example of a model of $\ZF+\DC_{<\kappa}$ which is not $\kappa$-closed in any model of $\ZFC$ with the same ordinals, which was pointed out in page~6 of \cite{Karagila:2018c} as a possible limitation of Theorem~3.5 there. This also answers, in part, question number 4 in Morris' thesis, showing that $\DC_{<\kappa}$ is not sufficient to prove that a model can be extended to a model of $\ZFC$ with the same ordinals.

\section{Is an even more general theorem possible?}
\begin{question}
Is it consistent that for every non-empty set $X$, there is a set $A_X$ which is a countable union of countable sets, and $\mathcal P(A_X)$ can be mapped onto $X$?
\end{question}
If a positive answer is at reach, first note that it is enough to consider $X$ which has the form $V_\alpha$ for some $\alpha$. One way of obtaining such result is by iterating similarly to the Morris construction, at each stage adding $V_\alpha$ subsets rather than $\omega_\alpha$. However, it is not immediately clear that this solution would work, or generalize properly as well.

This leads us to the following question, which is of independent interest and has importance to many other results in this field of choiceless set theory.
\begin{question}
Assume that $X$ is a set of ``regular cardinality'', namely there is no partition of $X$ into $<|X|$ sets of size $<|X|$.\footnote{It might be necessary to describe this in terms of surjections instead.} Is there a forcing which adds subsets to $X$ without upsetting the cardinal structure below $X$? What sort of limitations are necessary for such forcing to exist, and what are the necessary assumptions on it so that it does not collapse ``too many'' cardinals?
\end{question}
Many terms in this question are vaguely formulated, but for a good reason. There are very small details that reveal themselves only after considerable time working towards a solution, some of which might not even be known to us at this point. It is therefore better to remain vague.

To see why this is an interesting question, assume that $\kappa$ is a regular cardinal. If we force with $\Add(\kappa,1)$, we must add a bijection between $\kappa$ and $\kappa^{<\kappa}$. This implies a well-ordering of some initial segment of the universe was added. Is it possible to add a subset to $\kappa$ without adding any sets of rank $\alpha$ such that there is no surjection from $V_\alpha$ onto $\kappa$? A partial positive answer could be utilized in many ways to produce many choiceless results via iterations of fairly straightforward constructions.

\providecommand{\bysame}{\leavevmode\hbox to3em{\hrulefill}\thinspace}
\providecommand{\MR}{\relax\ifhmode\unskip\space\fi MR }
\providecommand{\MRhref}[2]{%
  \href{http://www.ams.org/mathscinet-getitem?mr=#1}{#2}
}
\providecommand{\href}[2]{#2}

\end{document}